\documentclass[12pt]{amsart}

\usepackage{graphics, amsmath, amsfonts, amssymb, amsthm, amscd, mathrsfs}
\newtheorem{theorem}{Theorem}[section]
\newtheorem{proposition}[theorem]{Proposition}
\newtheorem{lemma}[theorem]{Lemma}
\newtheorem{corollary}[theorem]{Corollary}

\theoremstyle{definition}
\newtheorem{definition}[theorem]{Definition}

\newtheorem{remark}[theorem]{Remark}
\numberwithin{equation}{section}

\newcommand{\A}{{\mathscr A}}
\newcommand{\C}{{\mathbb C}}
\newcommand{\N}{{\mathscr N}}
\renewcommand{\O}{{\mathscr O}}
\newcommand{\R}{{\mathbb R}}

\newcommand{\Aut}{{\operatorname{Aut}}}
\newcommand{\codim}{{\operatorname{codim}\,}}

\newcommand{\inv}{{^{-1}}}
\renewcommand{\sl}{{\operatorname{/\!\!/}}}
\newcommand{\pr}{{\operatorname{pr}}}
\newcommand{\ql}{{\operatorname{q\ell}}}
\newcommand{\lie}{\mathfrak}
\renewcommand{\phi}{\varphi}
\newcommand{\GL}{\operatorname{GL}}
\newcommand{\SL}{\operatorname{SL}}
\newcommand{\pt}{\partial}
\newcommand{\alg}{{\operatorname{alg}}}
\title{Sufficient conditions for holomorphic linearisation}

\author{Frank Kutzschebauch, Finnur L\'arusson, Gerald W.~Schwarz}

\address{Frank Kutzschebauch, Institute of Mathematics, University of Bern, Sidlerstrasse 5, CH-3012 Bern, Switzerland}
\email{frank.kutzschebauch@math.unibe.ch}

\address{Finnur L\'arusson, School of Mathematical Sciences, University of Adelaide, Adelaide SA 5005, Australia}
\email{finnur.larusson@adelaide.edu.au}

\address{Gerald W.~Schwarz, Department of Mathematics, Brandeis University, Waltham MA 02454-9110, USA}
\email{schwarz@brandeis.edu}



\date{January 11, 2016}

\thanks{F.~Kutzschebauch was partially supported by Schweizerischer Nationalfond grant 200021-140235/1.  F.~L\'arusson was partially supported by Australian Research Council grants DP120104110 and DP150103442.  F.~ L\'arusson and G.~W.~Schwarz would like to thank the University of Bern for hospitality and financial support and F.~Kutzschebauch and G.~W.~Schwarz would like to thank the University of Adelaide for hospitality and the Australian Research Council for financial support.}

\begin{document}

\begin{abstract}  
Let $G$ be a reductive complex Lie group acting holomorphically on  $X=\C^n$. The (holomorphic) Linearisation Problem asks if there is a holomorphic change of coordinates on $\C^n$ such that the $G$-action becomes linear.  Equivalently,  is there  a $G$-equivariant  biholomorphism $\Phi\colon X\to V$ where $V$ is a $G$-module? 
There is an intrinsic stratification of the categorical quotient $Q_X$, called the Luna stratification, where the strata are labeled by   isomorphism classes of representations of reductive subgroups of $G$. Suppose that there is a $\Phi$ as above. Then $\Phi$ induces a biholomorphism $\phi\colon Q_X\to Q_V$ which is stratified, i.e., the stratum of $Q_X$ with a given label is sent isomorphically to the stratum of $Q_V$ with the same label.

  The counterexamples to the Linearisation Problem   construct an action of $G$ such that $Q_X$ is not stratified biholomorphic to any $Q_V$. Our main theorem shows that, for most $X$, a stratified biholomorphism of $Q_X$ to some $Q_V$ is \emph{sufficient\/} for linearisation. In fact, we do not have to assume that $X$ is biholomorphic to $\C^n$, only that $X$ is a Stein manifold. 
\end{abstract}

\maketitle
\tableofcontents

\section{Introduction}  \label{s:introduction}

\noindent

The problem of linearising the action  of a reductive group $G$ on $\C^n$ has attracted much attention both in the algebraic and holomorphic settings (\cite{Huckleberry},\cite{KorasRussell},\cite{ Kraft1996}).  First there was work in the algebraic category.   If $X$ is an affine $G$-variety, then the quotient is the affine variety $Q_X$ with coordinate ring $\O_\alg(X)^G$. 
An early high point is a consequence of Luna's slice theorem \cite{Luna}. Suppose that $Q_X$ is a point and that $X$ is contractible. Then $X$ is algebraically $G$-isomorphic to a $G$-module. The structure theorem for the group of algebraic automorphisms of $\C^2$ shows that any action on $\C^2$ is linearisable \cite[Section 5]{ Kraft1996}.  As a consequence of a long series of results by many people, it was finally shown in \cite{KraftRussell} that an effective action of a positive dimensional $G$ on $\C^3$ is linearisable. The case of finite groups acting on $\C^3$ remains open.

The first counterexamples to the algebraic linearisation problem were constructed by Schwarz \cite{Schwarz1989} for $n\geq 4$.  His examples came from negative solutions to the equivariant Serre problem, i.e.,   there are algebraic $G$-vector bundles with base a $G$-module which are not  isomorphic to the trivial ones (those of the form $\pr\colon W\times W'\to W$ where $G$ acts diagonally on the $G$-modules $W$ and $W'$). It is interesting to note that in these counterexamples, the nonlinearisable actions may have the same stratified quotient  as a $G$-module. 

By the equivariant Oka principle of Heinzner and Kutzschebauch \cite{Heinzner-Kutzschebauch}, all holomorphic $G$-vector bundles over a $G$-module are trivial. Thus the algebraic counterexamples to linearisation are not counterexamples in the holomorphic category. But Derksen and Kutzsche\-bauch \cite{Derksen-Kutzschebauch} showed that for $G$ nontrivial, there is an $N_G\in\mathbb N$ such that there are nonlinearisable holomorphic actions of $G$ on $\C^n$, for every $n\geq N_G$.  

Consider Stein $G$-manifolds $X$ and $Y$. There are the categorical quotients $Q_X$ and $Q_Y$. We have the Luna stratifications of $Q_X$ and $Q_Y$ labeled by isomorphism classes of representations of reductive subgroups of $G$ (see Section \ref{sec:background}). We say that a biholomorphism  $\phi\colon Q_X\to Q_Y$ is \emph{stratified\/} if it sends the Luna stratum of $Q_X$ with a given label to the Luna stratum of $Q_Y$ with the same label.  If $\Phi\colon X\to Y$ is a $G$-biholomorphism, then the induced mapping $\phi\colon Q_X\to Q_Y$ is stratified. The counterexamples of Derksen and Kutzschebauch are actions on $\C^n$ whose   quotients are not isomorphic, via a stratified biholomorphism,  to the  quotient of a linear action.  We will show that, 
under a mild assumption, this is  the only way to get a counterexample to linearisation.  

Suppose that we have a stratified biholomorphism $\phi\colon Q_X\to Q_V$ where $V$ is a $G$-module. Then we may identify $Q_X$ and $Q_V$ and call the common quotient $Q$. We have quotient mappings $p\colon X\to Q$ and $r\colon V\to Q$. Assume there is an open cover $\{U_i\}_{i\in I}$ of $Q$ and $G$-equivariant biholomorphisms $\Phi_i:p^{-1}(U_i)\to r^{-1}(U_i)$ over $U_i$ (meaning that $\Phi_i$ descends to the identity map of $U_i$).  We express the assumption by saying that $X$ and $V$ are \textit{locally $G$-biholomorphic over a common quotient}.   
Equivalently, our original $\phi\colon Q_X\to Q_V$ locally lifts to $G$-biholomorphisms of $X$ to $V$.

Our first main   result  is the following.

\begin{theorem}\label{thm:main1}
Suppose that $X$ is a Stein $G$-manifold, $V$ is a $G$-module and $X$ and $V$ are locally $G$-biholomorphic over a common quotient. Then $X$ and $V$ are $G$-biholomorphic.
\end{theorem}

\begin{remark}
Assume that $X$ and $V$ are locally $G$-biholomorphic over a common quotient $Q$ and let $\{U_i\}$ and $\Phi_i$ be as above. The maps $\Phi_j\circ\Phi_i\inv$ give us an element of ${ H}^1(Q,\mathcal F)$,  where for $U$ open in $Q$, $\mathcal F(U)$ is the group of $G$-biholomorphisms of $r\inv(U)$ which induce the identity on $U$. Theorem \ref{thm:main1} says that the cohomology class associated to $X$ is trivial. On the other hand, given any element of ${ H}^1(Q,\mathcal F)$ one constructs a corresponding Stein $G$-manifold $X$ \cite[Theorem 5.11]{KLSOka}. Hence  the theorem is equivalent to the statement that ${H}^1(Q,\mathcal F)$ is trivial.
\end{remark}

 We now consider two conditions on $X$. Assume that the set of closed orbits with trivial isotropy group is open in $X$  and that the complement, a closed subvariety of $X$, has complex codimension at least two. We say that $X$ is \emph{generic\/}. Let $X_{(n)}$ denote the subset of $X$ whose isotropy groups have dimension $n$. Following the terminology of \cite{Schwarz1995} we say that $X$ is \emph{$2$-large\/} (or just  \emph{large\/}) if $X$ is generic and $\codim X_{(n)}\geq n+2$ for $n\geq 1$. For other conditions equivalent to $2$-largeness see \cite[Section 9]{Schwarz1995}.  The conditions ``generic'' and ``large'' hold for ``most'' $X$ (see Remark \ref{rem:generic} below).

 \begin{remark}
 In \cite[Corollary 14] {KLS} we 
 established Theorem \ref{thm:main1} with the extra hypothesis that $X$ (equivalently $V$) is generic. Removing the hypothesis is technically very difficult and requires the techniques of \cite{KLSOka}.
 \end{remark}
 
 Our second main result is the following.
 
\begin{theorem}\label{thm:main2}
Suppose that $X$ is a Stein $G$-manifold and $V$ is a $G$-module satisfying the following conditions.
\begin{enumerate}
\item There is a stratified biholomorphism $\phi$ from $Q_X$ to $Q_V$.
\item  $V$  (equivalently, $X$) is large.
\end{enumerate}
Then, by perhaps changing $\phi$, one can arrange that $X$ and $V$ are locally $G$-biholomorphic over $Q_X\simeq Q_V$, hence $X$ and $V$ are $G$-biholomorphic.
\end{theorem}

The proofs of our main theorems    use  the flows of the Euler vector field on $V$ and an analogous vector field on $X$. They give us smooth deformation retractions of $Q_X\simeq Q_V$ to a point which are covered by $G$-equivariant retractions of $X$ and $V$ to fixed points. The difficult part of the proof of Theorem \ref{thm:main2} is to construct local $G$-biholomorphic lifts of $\phi$. Then we can reduce to Theorem \ref{thm:main1}.

One can ask if assumption (2) of Theorem \ref{thm:main2} can be removed. In Section \ref{sec:small} we show 
it can  if $\dim Q\leq 1$ or $G=\SL_2(\C)$. We would be surprised if there is a counterexample to Theorem \ref{thm:main2} with (2) omitted.

\smallskip\noindent
\textit{Acknowledgement.}  We thank the referees for suggestions improving the exposition.

\section{Background}  \label{sec:background}

\noindent
We start with some background.  For more information, see \cite{Luna} and \cite[Section~6]{Snow}.  Let $X$ be a normal Stein space with a holomorphic action of a reductive complex Lie group $G$.  The categorical quotient $Q_X=X\sl G$ of $X$ by the action of $G$ is the set of closed orbits in $X$ with a reduced Stein structure that makes the quotient map $p\colon X\to Q_X$ the universal $G$-invariant holomorphic map from $X$ to a Stein space. When $X$ is understood, we drop the subscript $X$ in $Q_X$. Since $X$ is normal, $Q=Q_X$ is normal.  If $U$ is an open subset of $Q$, then $\O_X(p^{-1}(U))^G \simeq \O_Q(U)$. We say that a subset of $X$ is \textit{$G$-saturated\/} if it is a union of fibres of $p$. If $X$ is a $G$-module, then $Q$ is just the complex space corresponding to the affine algebraic variety with coordinate ring  $\O_\mathrm{alg}(X)^G$. 

If $Gx$ is a closed orbit, then the stabiliser (or isotropy group) $G_x$ is reductive.  We say that closed orbits $Gx$ and $G{y}$ have the same \textit{isotropy type} if $G_x$ is $G$-conjugate to $G_{y}$. Thus we get the \textit{isotropy type  stratification} of $Q$ with strata whose labels are  conjugacy classes of reductive subgroups of $G$. 

Assume that $X$ is smooth and let $Gx$ be a closed orbit. Then we can consider the \emph{slice representation\/} which is the action of $G_x$ on $T_xX/T_x(Gx)$. We say that closed orbits $Gx$ and $Gy$ have the same \emph{slice type\/} if they have the same isotropy type and, after arranging that $G_x=G_y$, the slice representations are isomorphic representations of $G_x$. 
The slice type (Luna) strata are locally closed smooth subvarieties of $Q$. The   Luna stratification  is finer than the isotropy type stratification, but the Luna strata are unions of connected components of the isotropy type strata \cite[Proposition 1.2]{Schwarz1980}. Hence if the isotropy strata are connected, the Luna strata and isotropy type strata are the same. This occurs for the case of a $G$-module \cite[Lemma 5.5]{Schwarz1980}. 
 Alternatively, one can show directly that in a $G$-module, the isotropy group of a closed orbit determines the slice representation (see \cite[proof of Proposition 1.2]{Schwarz1980}).

There is a unique open stratum $Q_\pr\subset Q$, corresponding to the closed orbits with minimal stabiliser. We call this the \emph{principal stratum\/} and the closed orbits above $Q_\pr$ are called \emph{principal orbits\/}. The isotropy groups of principal orbits are called \emph{principal isotropy groups\/}. By definition, $X$ is generic when the principal isotropy groups are trivial and the closed subvariety $p\inv (Q\setminus Q_\pr)$ has codimension at least 2 in $X$. Recall that $X$ is 
large if it is generic and $\codim X_{(n)}\geq n+2$ for $n\geq 1$.

\begin{remark}  \label{rem:generic}
 If $G$ is simple, then, up to isomorphism, all but finitely many $G$-modules $V$ with $V^G=0$ are 
 large \cite[Corollary 11.6 (1)]{Schwarz1995}.  The same result holds for semisimple groups but one needs to assume that every irreducible component of $V$ is a faithful module for the Lie algebra  of $G$ \cite[Corollary 11.6 (2)]{Schwarz1995}.  A \lq\lq random\rq\rq\ $\C^*$-module is 
 large, although infinite families of counterexamples exist.  More precisely, a faithful $n$-dimensional $\C^*$-module without zero weights is 
 large if and only if it has at least two positive weights and at least two negative weights and any $n-1$ weights have no common prime factor.  Finally, $X$ is 
 large if and only if every slice representation is 
 large and the property of being 
 large only depends upon the Luna stratification of $Q$. 
 \end{remark}

\section{Proof of Theorem \ref{thm:main1}}

\noindent
  We will need to show that some vector fields on $X$ are \emph{complete\/}, i.e., can be integrated for all real time $t$.
  
  \begin{lemma}\label{lem:integrate-near-F}
Let $A$ be a $G$-invariant holomorphic vector field on $X$ and let $F$ be a fibre of $p\colon X\to Q$. Then there is a $G$-saturated neighbourhood $U$ of $F$ and an $\epsilon>0$ such that the local $1$-parameter group $\psi_t$ of $A$ exists on $(-\epsilon,\epsilon)\times U$.
\end{lemma}

\begin{proof}
Let   $K$ be a maximal compact subgroup of $G$. By   \cite[Section 5 Corollary 1]{Heinzner-Kutzschebauch} there is a  $K$-orbit $O$ in $F$ and a neighbourhood basis $\{U_i\}_{i\in\mathbb N}$ of $O$ consisting of $K$-stable open sets with the following properties.
\begin{enumerate}
\item For all $i$, $G\cdot U_i$ is $G$-saturated.
\item For all $i$, any $K$-equivariant holomorphic map from $U_i$ to $X$ has a unique extension to a $G$-equivariant holomorphic map of $G\cdot U_i$ to $X$.
\end{enumerate}
Since $O$ is compact,   there is an $\epsilon>0$ and a neighbourhood $U$ of $O$ such that  $\psi_t$ is defined on $(-\epsilon,\epsilon)\times U$. Since $A$ is $K$-invariant and holomorphic, so is each $\psi_t$. We may assume that $U$ is one of the $U_i$. Then each $\psi_t$ extends  uniquely  to a $G$-equivariant holomorphic map $\widetilde\psi_t$ of $G\cdot U$ to $X$ and the $\widetilde\psi_t$ are easily seen to be a local $1$-parameter group corresponding to  $A$.  
\end{proof}

\begin{definition}  
Let $B$ be a holomorphic vector field on $Q$ (derivation  of $\O_Q$). We say that a $G$-equivariant holomorphic vector field $A$ on $X$ is a \emph{lift of $B$\/} if $A(p^*f)=p^*B(f)$ for every $f\in\O(Q)$. 
\end{definition}

We leave the proof of the following to the reader.

\begin{corollary}\label{cor:AB-complete}
Let $A$ be a $G$-invariant vector field on $X$ which is a lift of the holomorphic vector field $B$ on $Q$. Then $A$ is complete if and only if $B$ is complete.
\end{corollary}

Recall that we are assuming that the Stein $G$-manifold $X$ and the $G$-module $V$ are locally $G$-biholomorphic over their common quotient $Q$.

The scalar action of $\C^*$ on $V$ descends to a  $\C^*$-action on $Q$ (see below), 
in particular,  we have an action of $\R^{>0}=\{u\in\R\mid u>0\}$ on $Q$. We will now show how to lift  the $\R^{>0}$-action to $X$.  Let $\A_Q$ denote the sheaf of holomorphic vector fields on $Q$ and let $\A(Q)$ denote the global sections. Similarly we have the sheaf of holomorphic vector fields  $\A_X$ on $X$. Let $U$ be open in $Q$ and let $\A_X^G(U)$ denote $\A_X(p\inv(U))^G$. Then $\A_X^G$ is a coherent sheaf of $\O_Q$-modules  \cite{Roberts1986} as is $\A_Q$.  We have  $p_*\colon \A_X(p\inv(U))^G\to\A_Q(U)$  where $p_*(A)(f)=A(p^*(f))$ for $f\in\O_Q(U)\simeq\O_X(p\inv(U))^G$. Then  $p_*\colon\A_X^G\to\A_Q$ is a morphism of coherent sheaves of $\O_Q$-modules. Hence the kernel $\mathcal M$ of $p_*$ is coherent. 
 
 Let $r_1,\dots,r_m$ be homogeneous generators of $\O_\alg (V)^G$  where $r_i$ has degree $d_i$. Then $(r_1,\dots,r_m)\colon V\to\C^m$ induces a map $f\colon Q\to\C^m$ which is an algebraic isomorphism of $Q$ onto the image of $f$. Hence we can think of the quotient map $r\colon V\to Q$ as the polynomial map with entries $r_i$. Let $t\in\C^*$ and $q=(q_1,\dots,q_m)\in Q$. Define $t\cdot q= (t^{d_1}q_1,\dots,t^{d_m}q_m)$. This is a $\C^*$-action on $Q$ and   $r(tv)=t\cdot r(v)$, $t\in\C^*$, $v\in V$.
We have the Euler vector field $E=\sum_i x_i\pt/\pt x_i$ on $V$, where the $x_i$ are the coordinate functions on $V$. Let $y_1,\dots,y_m$ be the usual coordinate functions on $\C^m$. Then $r_*(E)=\sum d_i y_i\pt/\pt y_i\in\A(Q)$.

\begin{lemma}
Let $B=r_*E$. Then $B$ lifts to  a $G$-invariant holomorphic vector field $A$ on $X$.
\end{lemma}

\begin{proof}
Let $U$ be open and $G$-saturated in $V$ and let $\Phi\colon U\to \Phi(U)\subset X$ be a $G$-biholomorphism inducing the identity on $r(U)$. Let $A_U$ denote the image of $E|_U$ in $\A_X(\Phi(U))^G$ under the action of $\Phi$. Then $A_U$ is a lift of $B|_{r(U)}$. The various $A_U$  differ by elements  in the kernel $\mathcal M$ of $p_*$, hence a global lift of $B$ is obstructed by an element of  $H^1(Q,\mathcal M)$, which vanishes by Cartan's Theorem B. Hence $A$ exists.
\end{proof}

Choose a lift $A$ of $B$ and let $\psi_t$ denote the flow of $A$ on $X$. From Corollary \ref{cor:AB-complete} we have:

\begin{lemma}\label{lem:flowexists}
The flow $\psi_t$ exists for all $t\in\R$.
\end{lemma}

We have an $\R^{>0}$-action on $X$ where $t\cdot x=\psi_{\log t}(x)$, $t\in\R^{>0}$, $x\in X$, and it is the promised lift of the $\R^{>0}$-action on $Q$. We now need to find retractions of $Q$, $X$ and $V$. Choose positive integers $e_i$ such that $d_ie_i=d$ is independent of $i=1,\ldots, m$. For $q=(q_1,\dots,q_m)\in Q$ let $\rho(q)=\sum_i|q_i|^{2e_i}$. Choose $u\in \R^{>0}$ and set $Q_u=\{q\in Q\mid \rho(q)<u\}$. Let $h\colon [0,\infty)\to[0,u)$ be a diffeomorphism which is the identity in a neighbourhood  of $0$. Set 
$$
a(q)=\left(\frac{h(\rho(q))}{\rho(q)}\right)^{1/2d}\text{ and }\alpha(q)=a(q)\cdot q,\ q\in Q.
$$
Here $a(q)\cdot q$ denotes the $\C^*$-action. Now $\alpha$ is a diffeomorphism of $Q$ with $Q_u$ with inverse
$$
\beta(q)=b(q)\cdot q \text { where }b(q)=\left(\frac{h\inv(\rho(q))}{\rho(q)}\right)^{1/2d},\ q\in Q_u.
$$

\begin{proof}[Proof of Theorem \ref{thm:main1}] For $u>0$, let $X_u$ denote $p\inv(Q_u)$ and let $V_u$ denote $r\inv(Q_u)$. Choose $u>0$ so that we have a local $G$-biholomorphism $\Phi\colon X_u\to V_u$ inducing the identity on $Q_u$. Let $\eta_t$ be the flow of the Euler vector field on $V$. Then we have a $G$-diffeomorphism $\sigma$ of $V$ with $V_u$ which sends $v\in V$ to $\eta_{\log a(r(v))}(v)$. Using the flow $\psi_t$ of the vector field $A$ that we constructed on $X$, we have a $G$-diffeomorphism $\tau$ of $X$ with $X_u$ which sends $x\in X$ to $\psi_{\log a(p(x))}(x)$. By construction, $\sigma$ and $\tau$ map fibres $G$-biholomorphically to fibres and $\sigma\inv\circ\Phi\circ \tau$ is a $G$-diffeomorphism of $X$ and $V$ inducing the identity map on $Q$. By   \cite[Theorem 1.1]{KLSOka}, $X$ and $V$ are $G$-biholomorphic.
\end{proof}

\begin{remark}\label{rem:aboutA}
We used the fact that $X$ and $V$ are locally $G$-biholomorphic over $Q$ to construct our special $G$-invariant vector field $A$. But given 
any lift $A$ of $B=r_*E$, we can construct our  $G$-diffeomorphism which is biholomorphic on fibres, as long as we have a $G$-biholomorphism of neighbourhoods of $p\inv (r(0))$ and $r\inv(r(0))$ inducing the identity on $Q$.
\end{remark}

\section{Proof of Theorem \ref{thm:main2}}

\noindent
Assume for now  that we have a biholomorphism $\phi\colon Q_X\to Q_V$ which preserves the Luna stratifications. Note that $X^G$ is smooth and closed in $X$. We may identify $X^G$ with its image in $Q_X$ and similarly for $V^G$. Then $\phi$ induces a biholomorphism (which we also call $\phi$) from $X^G$ to $V^G$. We have $V=V^G\oplus V'$ where $V'$ is a $G$-module. Since $X^G\simeq V^G$ is  contractible, the normal bundle $\N(X^G)=(TX|_{X^G})/T(X^G)$   is a trivial $G$-vector bundle \cite{Heinzner-Kutzschebauch} and we have an isomorphism $\Phi\colon \N(X^G)\to V^G\times V'$ (viewing the 
target as the $G$-vector bundle $V^G\times V'\to V^G$). Since $TX|_{X^G}$ is also $G$-trivial, we may think of $\N(X^G)$ as a 
$G$-subbundle of $TX|_{X^G}$. Note that $\Phi$ restricts to $\phi$ on the zero section.

The following proposition does not need the hypothesis that $X$ or $V$ is large.

\begin{proposition}\label{prop:nearfixedpoints}
Let $\phi\colon X^G\to V^G$ and $\Phi$ be as above. Then there is a $G$-saturated neighbourhood $U$ of $X^G$ in $X$ and a $G$-saturated neighbourhood $U'$ of $V^G$ in $V$ and a $G$-biholomorphism $\Psi\colon U\to U'$ whose differential induces $\Phi$ on $\N(X^G)$. 
\end{proposition}

\begin{proof} Let $v_1,\dots,v_k$ be a basis of $V'$ and let $A_1,\dots,A_k$ denote the corresponding constant vector fields on $V$. Let $X_1,\dots,X_k$ denote their inverse images under $\Phi$. Then the $X_i$ are holomorphic vector fields defined on $X^G$ and they extend to global vector fields on $X$, which we also denote as $X_i$. Let $\rho^{(i)}_t$ denote the complex flow of $X_i$, $i=1,\dots,k$. For $(v,v')\in V^G\oplus V'$, $v'=\sum a_iv_i$, let 
$$
F(v,v')=\rho^{(1)}_{a_1}\rho^{(2)}_{a_2}\cdots\rho^{(k)}_{a_k}(\phi\inv(v)).
$$
Then $F$ is defined and biholomorphic on a neighbourhood of $V^G\times\{0\}$ and the derivative of $F$ along $V^G\times\{0\}$ is $\Phi\inv$.   The inverse of $F$ gives us a biholomorphism  $\Psi\colon U\to U'$ with the following properties:
\begin{enumerate}
\item $U$ is a neighbourhood of $X^G$ in $X$ and $U'$ is a neighbourhood of $V^G$ in $V$.
\item $\Psi$ restricts to $\phi$ on $X^G$.
\item $d\Psi$ restricted to $\N(X^G)$ gives $\Phi$.
\end{enumerate}
Let $K$ be a maximal compact subgroup of $G$. Averaging $\Psi$ over $K$ gives us a new holomorphic map (also called $\Psi$) which still satisfies the conditions  above, perhaps with respect to  smaller neighbourhoods $U_0$ and $U_0'$. Shrinking  we may assume that $U_0$ and $U_0'$ are $K$-stable. 
It follows from \cite[Section 5, Lemma 1, Proposition 1 and Corollary 1]{Heinzner-Kutzschebauch} that shrinking further we may achieve the following:
\begin{enumerate}
\addtocounter{enumi}{3}
\item The restriction of $\Psi$ to $U_0$ extends to a $G$-equivariant 
holomorphic map on $U=G\cdot U_0$ (which we also call $\Psi$).
\item The restriction of $\Psi\inv$ to $U_0'$ extends to a $G$-equivariant 
holomorphic map $\Theta$ on $U'=G\cdot U_0'$.
\end{enumerate}
We can reduce to the case  that $U$ and $U'$ are   connected.  Then it is clear that $\Psi\circ \Theta$ and $\Theta\circ \Psi$ are  identity maps.   Removing $p\inv(p(X\setminus U))$ from $U$ we can arrange  that it is $G$-saturated, and similarly for $U'$.
\end{proof}

Note that our $\Psi$ only induces $\phi$ on $X^G$.  Let $\psi$ denote the biholomorphism of $U\sl G$ and $U'\sl G$ induced by $\Psi$ and  
let $\tau$ denote $\phi\circ\psi\inv$. Then $\tau$ is a strata preserving biholomorphism defined on a neighbourhood of $V^G$ in $Q_V$ and $\tau$  is the identity on $V^G$. Our goal is to show that $\tau$ has a local $G$-biholomorphic lift to $V$ if we modify $\phi$ without changing its restriction to $X^G$. Let $x_0=\phi\inv(0)\in X^G$.

We will need to use some results from \cite{Schwarz2014}. We say that \emph{$X$ is admissible\/} if $X$ is generic and every holomorphic differential operator on $Q$ lifts to a $G$-invariant holomorphic differential operator on $X$. In particular, holomorphic vector fields on $Q$ lift to $G$-invariant holomorphic vector fields on $X$. Admissibility is a local condition over $Q$, hence $X$ is admissible if and only if each slice representation of $X$ is admissible. Since $X$ is large, 
\cite[Theorem 0.4]{Schwarz1995} and \cite[Remark 2.4]{Schwarz2013} show that $X$ is admissible. The only role of largeness in this paper is that it implies admissibility.

Let $s_1,\dots,s_n$ denote homogeneous invariant polynomials generating $\O_\mathrm{alg}(V')^G$, where $\delta_i$ is the degree of $s_i$, $i=1,\dots,n$. Let $s=(s_1,\dots,s_n)\colon V'\to\C^n$. We can identify $Q'=V'\sl G$ with the image of $s$ and we can identify $Q_V$ with $V^G\times Q'$. We have an action of $\C^*$ on $Q'$ where  $t\in \C^*$ sends $(q_1,\dots,q_n)\in Q'$ to $(t^{\delta_1}q_1,\dots,t^{\delta_n}q_n)$. 

Let   $\Aut_\ql(Q')$ denote the \emph{quasilinear\/}  automorphisms of $Q'$, that is, the automorphisms which commute with the $\C^*$-action.  An element of $\Aut_\ql(Q')$  is determined by its (linear) action on the invariant polynomials of degrees $\delta_i$, $i=1,\dots,n$. Hence $\Aut_\ql(Q')$ is  a linear algebraic group. Let $\sigma$ be a germ of a strata preserving automorphism of $Q'$ at $0=s(0)$. Then $\sigma(0)=0$. Let $\sigma_t$ denote the germ of an automorphism of $Q'$ which sends $q'$ to $t\inv\cdot\sigma(t\cdot q')$, $t\in\C^*$, $q'\in Q'$. It is not automatic that the limit of $\sigma_t$ exists as $t\to 0$. One needs to have the vanishing of certain terms of the Taylor series of $\sigma$ (see \cite[Section 2]{Schwarz2014}). But this occurs in the case that $V$ is   admissible \cite[Theorem 2.2]{Schwarz2014}. Moreover,  the limit  $\sigma_0$ lies in $\Aut_\ql(Q')$ and  $\sigma_t(q')$ is holomorphic in all $t\in\C$ and $q'\in Q'$ such that $\sigma_t(q')$ is defined. We consider $\Aut_\ql(Q')$ as the  subgroup of $\Aut(V^G\times Q')$ whose elements  send  $(v,q')$ to $(v,\rho(q'))$, $v\in V^G$, $q'\in Q'$, $\rho\in\Aut_\ql(Q')$.

Consider our automorphism $\tau$ defined on a neighbourhood of $V^G$ in $Q_V= V^G\times Q'$. Write $\tau=(\tau_1,\tau_2)$, where $\tau_1$ takes values in $V^G$ and $\tau_2$ in $Q'$. Then $\tau_1(v,0)=v$   and $\tau_2(v,0)=0$ for all $v\in V^G$. It follows from the inverse function theorem that $\tau_2(v,\cdot)$ is a germ of a strata preserving automorphism of $Q'$. Thus we have a holomorphic family of automorphisms $\tau_2(v,\cdot)_t$ with $\tau_{2,v}:=\tau_2(v,\cdot)_0\in\Aut_\ql(Q')$. 
Set $\tau_1(v,q)_t=\tau_1(v,t\cdot q)$. Then 
$$
\tau_t(v,q)=(\tau_1(v,q)_t,\tau_2(v,q)_t)
$$
is a homotopy connecting $\tau$ with $\tau_0$ where $\tau_0(v,q)=(v,\tau_{2,v}(q))$. The homotopy is holomorphic in all $v$, $q$ and $t$ such that $(v,t\cdot q)$ lies in the domain of $\tau$. The connected component of $\Aut_\ql(Q')$ containing $\tau_{2,v}$ is independent of $v$. Let $\rho\in\Aut_\ql(Q')$ denote any of the $\tau_{2,v}$, say $\tau_{2,0}$, which we can consider as an automorphism of $Q_V$. Change our original $\phi$ to  $\rho\inv\circ\phi$. This does not change $\phi$ restricted to $X^G$ or the biholomorphism $\Psi$ that we constructed. We then    find ourselves in the situation where $\tau_{2,v}$ lies in the identity component of $\Aut_\ql(Q')$ for every $v\in V^G$. Since $V$ is admissible,   the identity component of $\Aut_\ql(Q')$ is the image of $\GL(V')^G$ \cite[Proposition 2.8]{Schwarz2014}. Then there is a  neighbourhood of $0\in V^G$ on which we have a holomorphic lift  of the $\tau_{2,v}$ to elements of $\GL(V')^G$. Hence we can reduce to the case that $\tau_0$ is the identity. Now let $B=\{t\in \C: |t|\leq 
2\}$.  Shrinking our neighbourhood   $U'$ (which we now just consider as a neighbourhood of  $0\in V$), we can arrange that $\tau_t$ is defined on $\Omega:=U'\sl G$ for $t\in B$. Thus $\tau_t$ is 
a homotopy,  $t\in   B$, fixing  $(V^G\times\{0\})\cap\Omega$ and starting at the identity. Let 
$$
\Delta=\{(t,q)\in  B\times \Omega\mid q\in\tau_t(\Omega)\}.
$$
 Then $\Delta$ is a neighbourhood of $ B\times\{(0,0)\}$ where $(0,0)$ is the origin in $V^G\times Q'$. Hence there is a neighbourhood $\Omega_1$ of $(0,0)$ such that $  B\times\Omega_1\subset \Delta$. We have a (complex) time dependent vector field $C_t$  defined by
$$
C_t(\tau_t(q))=\frac {d\,\tau_s(q)}{ds}\big|_{t=s},\ t\in  B,\ q\in\Omega,
$$
and by definition of $\Omega_1$, $C_t(q)$ is defined for all $t\in B$ and $q\in\Omega_1$. Let 
$$
\Delta'=\{(t,q)\in[0,1]\times\Omega_1\mid \tau_t(q)\in\Omega_1\}.
$$
 Then, as before, $\Delta'$ contains an open set of the form $[0,1]\times\Omega_2$. On $\Omega_2$, $\tau_t$ is obtained by integrating the time dependent vector field $C_t$ where now we only consider $t\in[0,1]$.
Let $B^0$ denote the interior of $B$. 
Since $V$ is admissible,   we can lift (time-dependent) holomorphic vector fields on $B^0\times \Omega_1$ to $G$-invariant holomorphic vector fields on $B^0\times p\inv(\Omega_1)$. Hence $C_t$ lifts to a $G$-invariant holomorphic  vector field $A_t$ on $B^0\times p\inv (\Omega_1)$. Now we know that integrating $C_t$ on $\Omega_2$   lands us in $\Omega_1$ for $t\in[0,1]$. As in Lemma \ref{lem:flowexists} we can integrate $A_t$ for $t\in[0,1]$ and $x\in p\inv(\Omega_2)$ and we end up in $p\inv(\Omega_1)$. We thus have 
a homotopy whose value $\Theta$ at time 1 is a  $G$-biholomorphism of $p\inv(\Omega_2)$ which covers $\tau=\phi\circ\psi\inv$. Then $\Theta\circ\Psi$ is a $G$-biholomorphism inducing $\phi$ sending a $G$-saturated neighbourhood $U_X$ of $x_0$ onto a $G$-saturated neighbourhood $U_V$ of $0\in V$.

\begin{proof}[Proof of Theorem \ref{thm:main2}]
Let $E$ denote the Euler vector field on $V$. 
Since $X$ is admissible   we can lift the vector field $r_*E$ to a $G$-invariant vector field $A$ on $X$. Recall the $G$-equivariant flows $\eta_t$ of $E$ on $V$ and $\psi_t$ of $A$ on $X$. Let $X_u$ and $V_u$ be as before, $u>0$. Perhaps modifying $\phi$ by composition with an element of $\Aut_\ql(Q')\subset\Aut(Q_V)$, we can find a $G$-biholomorphism $\Phi\colon U_X\to U_V$   inducing $\phi$ as above. Then there is a $t\in\R$ such that $\psi_t(X_u)\subset U_X$ and $\eta_t(V_u)\subset  U_V$. The composition $\eta_{-t}\circ\Phi\circ\psi_t$ is a $G$-biholomorphism of $X_u$ with $V_u$ which induces $\phi$. Hence $X$ and $V$ are locally $G$-biholomorphic over a common quotient and we can apply Theorem \ref{thm:main1} or \cite[Corollary 14] {KLS}.
\end{proof}

\section{Small representations}\label{sec:small}

\noindent
Suppose that we have a strata preserving biholomorphism $\tau\colon Q_X\to Q_V$ as in Theorem \ref{thm:main2}. We know that $X$ and $V$ are $G$-equivariantly biholomorphic if $V$ is  large. In this section we investigate ``small''   $G$-modules $V$ which are not  large and see if we can still prove that $X$ and $V$ are $G$-equivariantly biholomorphic.
The proof of Theorem \ref{thm:main2} goes through if we can establish the following two statements where $V=V^G\oplus V'$ and $Q'=V'\sl G$.
\begin{enumerate}
\item[$(*)$]  Let $\phi$ be a germ  of a strata preserving automorphism near  the origin of $Q'$ and let $\phi_t=t\inv\circ\phi\circ t$. Then $\lim\limits_{t\to 0}\phi_t$ exists.
\item[$(**)$] Let $B$ be a holomorphic vector field on $Q'$ which preserves the strata, that is, $B(s)\in T_s(S)$ for every $s\in S$, where $S$ is any stratum of $Q'$. Then $B$ lifts to a $G$-invariant holomorphic vector field on $V'$.
\end{enumerate}

\begin{remark}\label{rem:samed}
Suppose that the minimal homogeneous generators of $\O_\alg(V')^G$ have the same degree. Then $\phi_0=\phi'(0)$ exists.
\end{remark}

The following theorem is one of the results in \cite{Jiang}.

\begin{theorem}
Suppose that $\dim Q\leq 1$. Then $X$ and $V$ are $G$-biholomorphic.
\end{theorem}

\begin{proof}
The case $\dim Q=0$ is an immediate consequence of Luna's slice theorem, so let us assume that $\dim Q=1$. Then $\O_\alg(V)^G$ is normal of dimension one, hence regular, and it is graded. Thus $\O_\alg(V)^G=\C[f]$, where $f$ is homogeneous and $Q\simeq\C$. First suppose that $Q$ has one stratum. Then the closed orbits in $V$ are the fixed points and   Proposition \ref{prop:nearfixedpoints} gives the required biholomorphism. The remaining case is where the strata of $Q$ are $\C\setminus\{0\}$ and $\{0\}$.  Then $(*)$ follows from Remark \ref{rem:samed}. As for $(**)$, our vector field is of the form $h(z)z\pt/\pt z$ where $h(z)$ is holomorphic. The vector field   lifts to an invariant holomorphic function times the Euler vector field on $V$.
\end{proof}

\begin{theorem}\label{thm:sl2}
Suppose that $G=\SL_2(\C)$. Then $X$ and $V$ are $G$-biholomorphic.
\end{theorem}

\begin{proof}
Let $R_d$ denote the representation of $G$ on $S^d\C^2$. Then the $G$-modules $V'$ where $(V')^G=0$ and $V'$ is  not  large are \cite[Theorem 11.9]{Schwarz1995} 
\begin{enumerate}
\item $kR_1$, $1\leq k\leq 3$.
\item $R_2$, $2R_2$, $R_2\oplus R_1$.
\item $R_3$, $R_4$.
\end{enumerate}
In all cases the quotient is $\C^k$ for some $k\leq 3$.
The cases $R_1$, $2R_1$, $R_2$, $R_3$ have quotient of dimension at most 1, hence they present no problem. Suppose that $V'=3R_1$. Then the generating invariants are determinants of degree 2, so we have $(*)$. Let $z_{ij}$ be the variable on $Q'=\C^3$ corresponding to the $i$th and $j$th copy of $\C^2$. Then the strata preserving vector fields are generated by the $z_{ij}\pt/\pt z_{k\ell}$. Thus we have 9 generators. But we have a canonical action of $\GL_3(\C)$ on $V'$ commuting with the action of $G$ and the image of $\lie {gl}_3(\C)$ is the span of the 9 generators. Hence we have $(**)$. For the case of $2R_2$ we have $(*)$ because the generators are polynomials of degree 2 and we have $(**)$ because $2R_2$ is an orthogonal representation \cite[Theorems 3.7 and 6.7]{Schwarz1980}. 

Suppose that $V'=R_4$. Then $V'$ is orthogonal, so $(**)$ holds. The quotient $Q'$ is isomorphic to the quotient of $\C^2$ by $S_3$, and it is known that strata preserving automorphisms have local lifts  \cite{Lyashko}, \cite[Theorem 5.4]{KrieglTensor}, hence we certainly have $(*)$. Finally, there is the case $V'=R_2\oplus R_1$. Then there are generating invariants homogeneous of degrees $2$ and $3$ and the zeroes of the degree $3$ invariant define the closure of the codimension one stratum. Thus we may think of $Q'$ as $\C^2$ with coordinate functions $z_2$ and $z_3$ where $z_i$ has weight $i$ for the action of $\C^*$. A strata preserving $\phi$ has to send $z_3$ to a multiple of $z_3$ (and fix the origin), so that $\phi=(\phi_2,\phi_3)$ where $\phi_3(z_2,z_3)=\alpha(z_2,z_3)z_3$. It follows easily that $\phi_0$ exists and we have $(*)$. The strata preserving vector fields must all vanish at the origin and preserve the ideal of $z_3$, so they are generated by $z_3\pt/\pt z_3$, $z_2\pt/\pt z_2$ and $z_3\pt/\pt z_2$. Since $V'$ is self dual, we can change the differentials of the generators $f_2$ and $f_3$ into invariant vector fields $A_2$ and $A_3$, and one can see that our three strata preserving vector fields below are in the span of the images of $A_2$, $A_3$ and the Euler vector field. Hence we have $(**)$.
\end{proof}

\bibliographystyle{amsalpha}
\bibliography{Oka.paperbib}

\providecommand{\bysame}{\leavevmode\hbox to3em{\hrulefill}\thinspace}
\providecommand{\MR}{\relax\ifhmode\unskip\space\fi MR }
\providecommand{\MRhref}[2]{%
  \href{http://www.ams.org/mathscinet-getitem?mr=#1}{#2}
}
\providecommand{\href}[2]{#2}
\begin{thebibliography}{KLM03}

\bibitem[DK98]{Derksen-Kutzschebauch}
Harm Derksen and Frank Kutzschebauch, \emph{Nonlinearizable holomorphic group
  actions}, Math. Ann. \textbf{311} (1998), no.~1, 41--53.

\bibitem[HK95]{Heinzner-Kutzschebauch}
Peter Heinzner and Frank Kutzschebauch, \emph{An equivariant version of
  {G}rauert's {O}ka principle}, Invent. Math. \textbf{119} (1995), no.~2,
  317--346.

\bibitem[Huc90]{Huckleberry}
Alan~T. Huckleberry, \emph{Actions of groups of holomorphic transformations},
  Several complex variables, {VI}, Encyclopaedia Math. Sci., vol.~69, Springer,
  Berlin, 1990, pp.~143--196.

\bibitem[Jia92]{Jiang}
Mingchang Jiang, \emph{On the holomorphic linearization and equivariant {S}erre
  problem}, Ph.D. thesis, Brandeis University, 1992.

\bibitem[KLM03]{KrieglTensor}
Andreas Kriegl, Mark Losik, and Peter~W. Michor, \emph{Tensor fields and
  connections on holomorphic orbit spaces of finite groups}, J. Lie Theory
  \textbf{13} (2003), no.~2, 519--534.

\bibitem[KLS]{KLSOka}
Frank Kutzschebauch, Finnur L\'arusson, and Gerald~W. Schwarz, \emph{Homotopy
  principles for equivariant isomorphisms}, preprint, arXiv:1503.00797.

\bibitem[KLS15]{KLS}
Frank Kutzschebauch, Finnur L{\'a}russon, and Gerald~W. Schwarz, \emph{An {O}ka
  principle for equivariant isomorphisms}, J. reine angew. Math. \textbf{706}
  (2015), 193--214.

\bibitem[KR04]{KorasRussell}
Mariusz Koras and Peter Russell, \emph{Linearization problems}, Algebraic group
  actions and quotients, Hindawi Publ. Corp., Cairo, 2004, pp.~91--107.

\bibitem[KR14]{KraftRussell}
Hanspeter Kraft and Peter Russell, \emph{Families of group actions, generic
  isotriviality, and linearization}, Transform. Groups \textbf{19} (2014),
  no.~3, 779--792.

\bibitem[Kra96]{Kraft1996}
Hanspeter Kraft, \emph{Challenging problems on affine {$n$}-space},
  Ast\'erisque (1996), no.~237, Exp.\ No.\ 802, 5, 295--317, S{\'e}minaire
  Bourbaki, Vol. 1994/95.

\bibitem[Lun73]{Luna}
Domingo Luna, \emph{Slices \'etales}, Sur les groupes alg\'ebriques, Soc. Math.
  France, Paris, 1973, pp.~81--105. Bull. Soc. Math. France, Paris, M\'emoire
  33.

\bibitem[Lya83]{Lyashko}
O.~V. Lyashko, \emph{Geometry of bifurcation diagrams}, Current problems in
  mathematics, {V}ol. 22, Itogi Nauki i Tekhniki, Akad. Nauk SSSR Vsesoyuz.
  Inst. Nauchn. i Tekhn. Inform., Moscow, 1983, pp.~94--129.

\bibitem[Rob86]{Roberts1986}
Mark Roberts, \emph{A note on coherent {$G$}-sheaves}, Math. Ann. \textbf{275}
  (1986), no.~4, 573--582.

\bibitem[Sch80]{Schwarz1980}
Gerald~W. Schwarz, \emph{Lifting smooth homotopies of orbit spaces}, Inst.
  Hautes \'Etudes Sci. Publ. Math. (1980), no.~51, 37--135.

\bibitem[Sch89]{Schwarz1989}
\bysame, \emph{Exotic algebraic group actions}, C. R. Acad. Sci. Paris S\'er. I
  Math. \textbf{309} (1989), no.~2, 89--94.

\bibitem[Sch95]{Schwarz1995}
\bysame, \emph{Lifting differential operators from orbit spaces}, Ann. Sci.
  \'Ecole Norm. Sup. (4) \textbf{28} (1995), no.~3, 253--305.

\bibitem[Sch13]{Schwarz2013}
\bysame, \emph{Vector fields and {L}una strata}, J. Pure Appl. Algebra
  \textbf{217} (2013), 54--58.

\bibitem[Sch14]{Schwarz2014}
\bysame, \emph{Quotients, automorphisms and differential operators}, J. Lond.
  Math. Soc. (2) \textbf{89} (2014), no.~1, 169--193.

\bibitem[Sno82]{Snow}
Dennis~M. Snow, \emph{Reductive group actions on {S}tein spaces}, Math. Ann.
  \textbf{259} (1982), no.~1, 79--97.

\end{thebibliography}

\end{document}